\documentclass[12pt]{article}

\usepackage{setspace}
\usepackage{amsthm, amssymb, amstext}
\usepackage[fleqn]{amsmath}
\usepackage{latexsym}
\usepackage[dvips]{graphicx}
\usepackage{comment}
\usepackage{hyperref}
\usepackage{mathtools}
\usepackage{enumerate} 

\usepackage{todonotes}
\usepackage{comment}

\def\A{\mathcal{A}}
\def\B{\mathcal{B}}

\def\F{\mathcal{F}}

\def\N{\mathcal{N}}
\def\S{\mathcal{S}}

\newtheorem{theorem}{Theorem}[section]
\newtheorem{lemma}{Lemma}[section]
\newtheorem{proposition}{Proposition}[section]

\newtheorem{corollary}{Corollary}[section]

\newtheorem{claim}{Claim}
\newtheorem{problem}{Problem}

\theoremstyle{definition}

\usepackage{amsthm}
\let\oldproofname=\proofname
\renewcommand{\proofname}{\rm\bf{\oldproofname}}

\title{Enclosings of Decompositions of Complete Multigraphs in $2$-Factorizations}
\date{}
\author{Carl Feghali and Matthew Johnson \\ School of Engineering and  Computing Sciences, \\ Durham University\\
{\small \texttt{\{carl.feghali,matthew.johnson2\}@durham.ac.uk}}
}

\begin{document}
\maketitle

\begin{abstract}
Let $k$, $\lambda$ and $\mu$ be positive integers.  A decomposition of a multigraph $ \lambda G$ into edge-disjoint subgraphs $G_1, \ldots , G_k$ is said to be \emph{enclosed} by a decomposition of a multigraph $\mu H$ into edge-disjoint subgraphs $H_1, \ldots , H_k$ if $\mu > \lambda$ and $G_i$ is a subgraph of $H_i$, $1 \leq i \leq k$.   In this paper we initiate the study of when a decomposition can be enclosed by a decomposition that  consists of spanning subgraphs. 

A decomposition of a graph is a 2-factorization if each subgraph is 2-regular and is Hamiltonian if each subgraph is a Hamiltonian cycle. 
Let $n$ and $m$ be positive integers. We give necessary and sufficient conditions for enclosing a decomposition of $\lambda K_n$ in a $2$-factorization of $\mu K_{n+m}$ whenever $\mu>\lambda$ and $m \geq n-2$. We also give necessary and sufficient conditions for enclosing a decomposition of $\lambda K_n$ in a Hamiltonian decomposition of $\mu K_{n+m}$ whenever $\mu > \lambda$ and $m \geq n-1$, or $\mu > \lambda$, $n=3$ and $m=1$, or $\mu = 2$, $\lambda=1$ and  $m=n-2$. 
\end{abstract}

\section{Introduction}

In this paper, graphs are undirected and may contain multiple edges and loops.    Our chief object of study is the complete graph on $n$ vertices in which each edge has multiplicity $\lambda$.  We denote this graph by $\lambda K_n$.

The set of vertices and edges of a graph $G$ are denoted by $V(G)$ and~$E(G)$ respectively.  An edge that joins a pair of vertices $u$ and $v$ is called a $uv$-edge. A \emph{decomposition} of size $k$ of a graph $G$ is a collection $\{G_1, \dots, G_k\}$ of spanning subgraphs of $G$ such that $E(G) = \bigcup_{i=1}^n E(G_i)$ and $E(G_i) \cap E(G_j) = \emptyset$ if $i \not= j$.   A \emph{partial decomposition} of a graph $G$ is a decomposition of some subgraph of $G$.   A \emph{$2$-factor} of a graph $G$ is a $2$-regular spanning subgraph of $G$. A \emph{$2$-factorization} of a graph $G$ is a decomposition of $G$ into $2$-factors. If, in addition, the $2$-factors are connected, then the decomposition  is called a \emph{Hamiltonian decomposition} and the $2$-factors are \emph{Hamiltonian cycles}.  In a \emph{path decomposition} each subgraph in the decomposition is the union of disjoint paths and cycles; it is \emph{strong} if, in fact, none of the subgraphs contain a cycle.     Let us emphasize that although we consider each subgraph in a decomposition of a graph $G$ to span $G$, the subgraph can contain isolated vertices (so in the case of path decompositions some of the paths might be trivial and contain no edges).
 
Let $k$, $\lambda$ and $\mu$ be positive integers. A decomposition $\{G_1, \dots, G_k\}$ of a multigraph $\lambda G$ is said to be \emph{enclosed} in a decomposition $\{H_1, \dots, H_k\}$ of $\mu H$ if $\mu > \lambda$ and $G_i$ is a subgraph of $H_i$, $1 \leq i \leq k$.

We state two enclosing problems.  The first, on Hamiltonian decompositions, was posed by Bahmanian~\cite{bathesis}.  The second is the analogue for 2-factorizations.

\begin{problem}\label{problem:hamilton}
Let $n$, $\lambda$ and $\mu$ be positive integers such that $\mu > \lambda$, and let $m$ be a non-negative integer. Find necessary and sufficient conditions for enclosing a decomposition of $\lambda K_n$ in a Hamiltonian decomposition of $\mu K_{n+m}$.
\end{problem}

\begin{problem}\label{problem:2factor}
Let $n$, $\lambda$ and $\mu$ be positive integers such that $\mu > \lambda$, and let~$m$ be a non-negative integer. Find necessary and sufficient conditions for enclosing a decomposition of $\lambda K_n$ in a $2$-factorization of $\mu K_{n+m}$.
\end{problem}

The analogue of Problem~\ref{problem:hamilton}  where $\mu = \lambda = 1$ was solved by Hilton~\cite{hilton1} using the technique of \emph{amalgamations}. Let us remark that such a problem, where the number of vertices increases but the edge multiplicity does not, are known as \emph{embedding} problems; that is, the decomposition of the smaller graph is said to be \emph{embedded} in the decomposition of the larger graph. In the situation where the decomposition sought consists of spanning subgraphs, several other embedding problems have also been solved using amalgamations; see, for example,~\cite{amalgamation4, hilton1, amalgamation1, amalgamation2, amalgamation3}.  We shall briefly use this technique in Section~\ref{section:amalgamation}.

We present two main theorems that provide extensive solutions to 
Problems~\ref{problem:hamilton} and~\ref{problem:2factor} leaving open only some cases where $m<n$.  For a decomposition $\A$ of a graph $G$, we denote by $\S_i(\A)$ the set of subgraphs in the decomposition that contain exactly $i$ edges, and, for $u,v \in V(G)$, $\S_1(u,v,\A)$  denotes the subset of $\S_1(\A)$ that contains subgraphs that each contain only a $uv$-edge.

\begin{theorem}\label{thm:hamilton}
Let $n$, $m$,  $\mu$, $\lambda$ and $k$ be positive integers where 
\begin{enumerate}[{\normalfont (i)}]
\item $\mu > \lambda$ and $m \geq n-1$, or  
\item $\lambda=1$, $\mu=2$ and $m=n-2$, or 
\item $\mu > \lambda$, $n=3$ and $m=1$. 
\end{enumerate}
Then a strong path decomposition $\A$ of $\lambda K_n$ of size $k$ can be enclosed in a Hamiltonian decomposition of $\mu K_{n+m}$ if and only if
\begin{enumerate}[{\normalfont (M1)}]
\item $k=  {\displaystyle \frac{\mu(m+n-1)}{2}}$, 
\item $\displaystyle{\sum_{i=0}^{n-m}{(n-m - i)|\S_i(\A)|} \leq (\mu - \lambda) {\displaystyle \frac{n(n-1)}{2}}}$, and
\item if {\normalfont (ii)} or {\normalfont (iii)} holds, then for each $u, v \in V(K_n)$,\\  $|\S_1(u,v, \A)| \leq (\mu - \lambda) \left({\displaystyle \frac{n(n-1)}{2} - 1}\right) - |\S_0(\A)|$.
\end{enumerate}
\end{theorem}

\begin{theorem}\label{thm:2factor}
Let $n$, $m$, $\mu$, $\lambda$ and $k$ be positive integers where $\mu > \lambda$ and $m \geq n-2$. Then a path decomposition $\A$ of $\lambda K_n$ of size $k$ can be enclosed in a $2$-factorization of $ \mu K_{n+m}$ if and only if

\begin{enumerate}[{\normalfont (N1)}]
\item $k=  {\displaystyle \frac{\mu(m+n-1)}{2}}$, 
\item the number of subgraphs in $\A$ that are 2-factors of $\lambda K_n$ is at most $\displaystyle{\frac{\mu(m-1)}{2}}$, and
\item $\displaystyle{\sum_{i=0}^{n-m}{(n-m - i)|\S_i(\A)|} \leq (\mu - \lambda) {\displaystyle \frac{n(n-1)}{2}}}$.
\end{enumerate}
\end{theorem}

To the best of our knowledge, our paper is the first to consider enclosing problems where the subgraphs of the decompositions sought are spanning.  Past work has focused on decompositions in which each subgraph is isomorphic to a cycle of one fixed length~\cite{ Asplund0, Asplund2, asplundone, c31, c32, c33, c42, c4} or where each subgraph is isomorphic to a cycle of one of a number of fixed lengths~\cite{ceven, cvarying}.

In the rest of the paper we assemble the tools that will, finally, in Section~\ref{section:main}, be used to prove Theorems~\ref{thm:hamilton} and~\ref{thm:2factor}. In Section~\ref{section:extend}, we find necessary and sufficient conditions for a decomposition of $\lambda K_n$ to be enclosed by a decomposition of $\mu K_n$ in which each subgraph contains at least some specified number of edges, and then in Section~\ref{section:amalgamation} we formulate conditions for the existence of solutions to Problems~\ref{problem:hamilton} and ~\ref{problem:2factor} in terms of the number of edges that each subgraph must contain if it is to be enclosed. In Section~\ref{section:lemmas}, two useful lemmas are given which guarantee when it is possible to extend partial path decompositions.

 \section{Extendibility}\label{section:extend}
 
Let $k$ be a positive integer.  A $k$-edge-colouring of a graph $G$ is a mapping from $E(G)$ to a set of $k$ colours, and, for each colour, the set of edges assigned that colour form a \emph{colour class}.  For each colour class, one can consider the graph with vertex set $V(G)$ and edge set equal to the colour class, and, in fact, we shall abuse terminology and also refer to this graph as a colour class.
Thus an edge-colouring can be thought of as a \emph{decomposition} of $G$ where the subgraphs of the decomposition are the colour classes.  And similarly when given a decomposition, we can refer to its subgraphs as colour classes (as we shall do henceforth). 
At the price of giving two ways to define a decomposition, the results and proofs in this paper can be expressed more simply.
 Let us emphasise that
\begin{itemize}
\item our edge-colourings are not necessarily proper, and
\item a colour class might contain no edges.
\end{itemize}

If $\A$ is a decomposition of a graph $G$ of size $k$ then its colour classes can be denoted $\{\A(1), \dots, \A(k)\}$ or $\{G(1), \dots, G(k)\}$ (it is convenient to have this choice of notation as sometimes are focus will be on $\A$, at other times on $G$).

Let $G$ be a graph, and let $H$ be a supergraph of $G$.  We write this as $G \subseteq H$, and $H \setminus G$ denotes the graph obtained from $H$ by deleting all the edges of $G$.  Let $\A$ be a  path decomposition of $G$ with $k$ colour classes, and let $\alpha$ be a positive integer. Then $G$ is said to be \emph{$\alpha$-extendible} with respect to $(\A, H)$ if there exists a graph $F \subseteq H \setminus G$ and a path decomposition $\A^*$ of $G \cup F$ such that
 \begin{itemize}
 \item the restriction of $\A^*$ to $G$ is $\A$, and
 \item for $1 \leq i \leq k$, the number of edges in  $\A^*(i)$ is at least $\alpha$.
 \end{itemize}
 If, in addition, $\A$ and $\A^*$ are required to be strong path decompositions then~$G$ is said to be \emph{strongly} $\alpha$-extendible with respect to $(\A, H)$.
  
 Let $n$, $\lambda$ and $\mu$ be positive integers such that $\mu > \lambda$, and let $\A$ be a path decomposition of $\lambda K_n$.  Suppose that $\lambda K_n$ is $\alpha$-extendible with respect to $(\A, \mu K_n)$. We have that  
\begin{equation}\label{necessary-e}\tag{$\ast$}
 \sum_{i=0}^{\alpha-1} (\alpha - i)|\S_i(\A)| \leq (\mu - \lambda) \frac{n(n-1)}{2},
 \end{equation}
 since the left hand side counts the number of edges that must be added to the colour classes of $\A$ and the right hand side is the number of edges in $\mu K_n \setminus \lambda K_n$.

The following easy propositions state that, for $\alpha = 1, 2$,~(\ref{necessary-e}) is a sufficient condition for $\lambda K_n$ to be $\alpha$-extendible with respect to $(\A, \mu K_n)$.

\begin{proposition}\label{prop:1}
Let $\lambda$ and $\mu$ be positive integers such that $\mu > \lambda$, and let~$\A$ be a path decomposition of $\lambda K_n$. Then $\lambda K_n$ is $1$-extendible with respect to $(\A, \mu K_n)$ if and only if 
\[
|\S_0(\A)| \leq (\mu - \lambda) \frac{n(n-1)}{2},
\]
and if $\A$ is strong then $\lambda K_n$ is, in addition, strongly 1-extendible.
\end{proposition}

\begin{proof}
This follows immediately from the fact that to each colour class in $\S_0(\A)$ one can add a distinct edge from $\mu K_n \setminus \lambda K_n$ to obtain the needed decomposition of $\mu K_n$. 
\end{proof}

\begin{proposition}\label{prop:3}
Let $\lambda$ and $\mu$ be positive integers such that $\mu > \lambda$, and let~$\A$ be a  path decomposition of $\lambda K_n$. Then $\lambda K_n$ is $2$-extendible with respect to $(\A, \mu K_n)$ if and only if 
\[
2|\S_0(\A)| + |\S_1(\A)| \leq (\mu - \lambda) \frac{n(n-1)}{2}.
\]
\end{proposition}

\begin{proof}
This follows from the fact that we need to add, from the edges of $\mu K_n \setminus \lambda K_n$, two edges for each colour class of $\S_0(\A)$, and one edge for each colour class of $\S_1(\A)$. 
\end{proof}

Notice that Propositions~\ref{prop:1} and~\ref{prop:3} differ in that the latter contains no statement about \emph{strong} path decompositions.  We cannot immediately adapt Proposition~\ref{prop:3} to say that if $\A$ is strong, then it is strongly 2-extendible as this requires not just that there are enough edges but that they can be added to the colour classes in such a way that cycles (of size two) are not created.  So we turn now to the problem of finding necessary and sufficient conditions for a strong path decomposition $\A$ of $\lambda K_n$ to be strongly $2$-extendible with respect to $(\A, \mu K_n)$. Let us first introduce some  notation and terminology.   For distinct $u, v \in V(K_n)$, recall that $\S_1(u,v, \A)$ denotes the subset of colour classes in $\S_1(\A)$  that contain an edge with endpoints $u$ and $v$.  
As $\S_1(\A)$ and each $\S_1(u,v,\A)$ is a set of colour classes that each contain a single edge, we will think of each of them as just a set of edges.   A 2-path is a pair of edges that do not join the same pair of vertices (they might or might not be adjacent); clearly in a strong path decomposition, any colour class with exactly two edges is a 2-path.
 We say that $\S_1(\A)$ is \emph{addible} with respect to $ \mu K_n$ if there exists an injection $\phi$ between $S_1(\A)$ and the edges of $\mu K_n \setminus \lambda K_n$ such that the two edges $e$ and $\phi(e)$ form a 2-path (that is, $e$ and $\phi(e)$ do not form a 2-cycle).  

\begin{lemma}\label{lem:addible}
Let $\lambda$, $\mu$ and $n$ be positive integers such that $\mu > \lambda$, and let $\A$ be a path decomposition of $\lambda K_n$. Then $\S_1(\A)$ is addible with respect to $\mu K_n$ if and only if 
\begin{enumerate}[{\normalfont ({A}1)}]
\item $\S_1(\A) \leq (\mu - \lambda)\displaystyle{\frac{n(n-1)}{2}}$, and
\item for each $u, v \in V(K_n)$, $\S_1(u,v, \A) \leq (\mu - \lambda)\left(\displaystyle{\frac{n(n-1)}{2}} - 1\right)$.
\end{enumerate}
\end{lemma}

\begin{proof}
As the function $\phi$ required must be injective, the necessity of the two conditions follows immediately from, respectively, the requirement that $|\S_1(\A)|$ is at most the number of edges in $\mu K_n \setminus \lambda K_n$, and the requirement that each $|\S_1(u,v,\A)|$ is at most the number of edges in $\mu K_n \setminus \lambda K_n$ that are not $uv$-edges.

To prove sufficiency we describe how to find a suitable function $\phi$.  In order to carry out this task, we create a network $\N$. It has vertex set $\{s\} \cup \{t\} \cup X \cup Y$ where $s$ is a source, $t$ is a sink, $X = \{x_{\{u, v\}}: \{u, v\} \subseteq V(K_n)\}$ and $Y=  \{y_{\{u, v\}}: \{u, v\} \subseteq V(K_n) \}$.  For simplicity, we will write $x_{u,v}$ and $y_{u,v}$ without braces but we note that $x_{u,v}$ and  $x_{v, u}$ are the same vertex. The set of arcs of $\N$ contains the following:
\begin{itemize}
\item a \emph{start} arc $(s, x_{u, v})$ for each $x_{u, v} \in X$,
\item an \emph{end} arc $(y_{u, v}, t)$ for each $y_{u, v} \in Y$, and
\item an \emph{intermediate} arc $(x_{u, v}, y_{w, z})$ if and only if $\{u, v\} \not= \{w, z\}$. 
\end{itemize}

\noindent
Each arc is assigned a capacity:  
\begin{itemize}
\item each start arc $(s, x_{u, v})$ has capacity $|\S_1(u,v, \A)|$, and
\item each intermediate and end arc has capacity $\mu - \lambda$.  
\end{itemize}

Let $C$ be a cut that separates $s$ from $t$.  We shall show that $C$ has capacity at least $|\S_1(\A)|$.  If every vertex in $Y$ is incident with at least one arc in the cut, this follows from (A1), since $|Y|=\frac{n(n-1)}{2}$ and all arcs incident with $Y$ have capacity $\mu-\lambda$.  Otherwise we can suppose that $C$ contains no edges incident with $y_{w,z}$ for some $w$ and $z$.  As $y_{w,z}$ is adjacent to $t$ and to every vertex in $X$ except $x_{w,z}$, we must have that $C$ contains $(s,x_{u,v})$ for all $\{u,v\} \neq \{w,z\}$, and the capacity of these arcs is $|\S_1(\A)|-|\S_1(w,z,\A)|$.
So if $C$ also contains $(s,x_{w,z})$, then we immediately have the total capacity of $C$ is at least $|\S_1(\A)|$.  If $C$ does not contain $(s,x_{w,z})$, it must contain, for all $\{u,v\} \neq \{w,z\}$, either $(x_{w,z}, y_{u,v})$ or $(y_{u,v},t)$.  As there are $\frac{n(n-1)}{2}-1$ such pairs $\{u,v\}$ and all these arcs have capacity $\mu-\lambda$ we have, in this case, that the capacity of $C$ is at least
\[
|\S_1(\A)|-|\S_1(w,z,\A)|+ (\mu - \lambda)\left(\displaystyle{\frac{n(n-1)}{2}} - 1\right),
\]
which, by (A2), is at least $|\S_1(\A)|$.

Thus, by the celebrated Max-Flow Min-Cut Theorem (see, for example,~\cite{diestel}), there is a flow $f$ in $\N$ of size $|\S_1(\A)|$ and, as the capacities are integers, we can assume the flow along each arc has an integer value.  It is clear that each $f(s,x_{u,v})$ must be $|\S_1(u,v,\A)|$ so the flow out of each $x_{u,v}$ must also be $|\S_1(u,v,\A)|$.  Hence we can think of the flow out of $x_{u,v}$ being divided up into units that can be bijectively matched with the edges of $|\S_1(u,v,\A)|$, and then for each $e \in \S_1(u,v,\A)$, if $e$ is matched with a unit of flow to $y_{w,z}$, we let $\phi(e)$ be an edge from $w$ to $z$.  As there is no arc from $x_{u,v}$ to $y_{u,v}$, we have that $e$ and $\phi(e)$ do not join the same pair of edges, and as the flow into each $y_{w,z}$ is at most $\mu - \lambda$ (as that is the capacity of the only outgoing arc), we have that there are at most $\mu-\lambda$ edges $e$ for which $\phi(e)$ is an edge from $w$ to $z$ so we can assume that $\phi$ is injective.
\end{proof}

We are now able to prove a result on when a decomposition of $\lambda K_n$ is strongly 2-extendible.

\begin{proposition}\label{prop:2}
Let $\lambda$ and $\mu$ be positive integers such that $\mu > \lambda$, and let $\A$ be a strong path decomposition of $\lambda K_n$. Then $\lambda K_n$ is strongly $2$-extendible with respect to $(\A, \mu K_n)$ if and only if
\begin{enumerate}[{\normalfont ({B}1)}]
\item $2|\S_0(\A)| + |\S_1(\A)| \leq (\mu - \lambda) \displaystyle{\frac{n(n-1)}{2}}$, and
\item for each $u, v \in V(K_n)$,  $|\S_0(\A)|+|\S_1(u,v, \A)| \leq (\mu - \lambda)\left(\displaystyle{\frac{n(n-1)}{2} - 1}\right)$. 
\end{enumerate}
\end{proposition}

\begin{proof}
\textit{Necessity}: Suppose that $\lambda K_n$ is strongly $2$-extendible with respect to $(\A, \mu K_n)$. Let $\A'$ be the resulting partial strong path decomposition in which each colour class contains at least two edges. Then (\ref{necessary-e}) implies that (B1) holds.  Moreover, note that each colour class of $\S_1(u,v, \A)$ must be extended to a colour class of $\A'$ with at least one more edge (that cannot be a $uv$-edge), and each colour class of $\S_0(\A)$ must be extended to a colour class of $\A'$ with at least two edges (at most one of which is a $uv$-edge).  Thus (B2) follows as the right hand side counts the number of edges of $\mu K_n \setminus \lambda K_n$ that are not $uv$-edges.

\textit{Sufficiency}: We must prove that $\lambda K_n$ is strongly $2$-extendible with respect to $\A$ if (B1) and (B2) hold.  So we need to show that we can add, from the edges of $\mu K_n \setminus \lambda K_n$,
\begin{itemize}
\item for each colour class of $\S_0(\A)$, a 2-path, and
\item for each edge $e \in \S_1(\A)$, an edge $e'$ such that $e$ and $e'$ form a 2-path.
\end{itemize}
In other words, we need to find two sets $E_0$ and $E_1$ using distinct edges of $\mu K_n \setminus \lambda K_n$ such that
\begin{itemize}
\item $E_0$ contains $|\S_0(\A)|$ 2-paths, 
\item $E_1$ contains $|\S_1(\A)|$ edges and there is a bijection $\phi: \S_1(\A) \rightarrow E_1$ such that, for all $e \in \S_1(\A)$, $e$ and $\phi(e)$ form a 2-path. 
\end{itemize}
Notice that conditions (A1) and (A2) of Lemma~\ref{lem:addible} are satisfied as they are weaker conditions than (B1) and (B2). Thus Lemma~\ref{lem:addible} implies $\S_1(\A)$ is addible with respect to $\mu K_n$ and we can find $E_1$ and $\phi$ as required.  It only remains to construct $E_0$.  We initially assume it is the empty set; then our aim is to add 2-paths until $|E_0|=|\S_0(\A)|$.  Possibly, while doing this we will amend $E_1$ and $\phi$, but we will not reduce the size of $E_1$, and will take care to ensure that each pair $e$ and $\phi(e)$ is always a 2-path.
Let $L$ be the edges of $\mu K_n \setminus \lambda K_n$ in neither $E_0$ nor $E_1$ (so as we construct $E_0$, $L$ will also be continually redefined as edges of $L$ are used to enlarge $E_0$ and $E_1$ is amended).  For any pair of vertices $u$ and $v$, let $L(u,v)$ be the edges of $L$ between $u$ and $v$.

First choose $x, y \in V$ such that $|L(x,y)| \geq |L(u,v)|$ for any $u, v \in V(K_n)$.  Let $L'=L \setminus L(x,y)$. 
Add to $E_0$ a maximal set of disjoint 2-paths using edges of $L'$.  If we are not done, then any unused edges of $L'$ must all join the same pair $u,v$ of vertices (else further 2-paths can be found).  Next add to $E_0$ as many 2-paths as possible (or as needed) that use one $uv$-edge and one $xy$-edge from $L$.  By the initial choice of $x$ and $y$, when no further such 2-paths can be added, we have that $L$ contains only $xy$-edges.

Let us show that if $|L|<2$, then we are done.  If we are not done then the number of 2-paths in $E_0$ is less than $|\S_0(\A)|$ so the number of edges used in these 2-paths is at most $2(|\S_0(\A)|-1)$.  The number of edges used in $E_1$ is $|\S_1(\A)|$.  If $|L|<2$, then all but at most one edge of $\mu K_n \setminus \lambda K_n$ is in either $E_0$ or $E_1$.  Hence
\[
 (\mu - \lambda) \frac{n(n-1)}{2} -1 \leq 2|E_0| + |E_1| \leq   2(|\S_0(\A)|-1) + |\S_1(\A)|,
\]
contradicting (B1).  

Let $E_1(\overline{x,y})$ be the edges of $E_1$ that are not $xy$-edges.
As long as we need to add further 2-paths to $E_0$, we do, if possible, one of the following (remembering that $L$ is a set of $xy$-edges of size at least 2):
\begin{itemize}
\item if there is a 2-path $P$ in $E_0$ that does not contain an $xy$-edge, then we remove $P$ from $E_0$ and replace it with two 2-paths which each use an edge of $P$ and an $xy$-edge from $L$;
\item if there is an edge $e \in E_1(\overline{x,y})$ and $f = \phi^{-1}(e)$ is not an $xy$-edge, then we remove $e$ from $E_1$ and replace it with an $xy$-edge $e'$ from $L$ (and let $\phi(f)=e'$), and also add to $E_0$ a 2-path consisting of $e$ and an $xy$-edge from $L$.
\end{itemize}

Suppose that neither operation is possible and still $|E_0| < |\S_0(\A)|$.  We notice that every 2-path in $E_0$ contains exactly one edge that is not an $xy$-edge, and, for every edge $e$ in $E_1(\overline{x,y})$, we have $\phi^{-1}(e) \in \S_1(x,y,\A)$.  Thus, as all the edges of $\mu K_n \setminus \lambda K_n$ that are not $xy$-edges belong to either $E_0$ or $E_1(\overline{x,y})$, we have
\[
|\S_0(\A)|+ |\S_1(x,y, \A)| > |E_0| + |E_1(\overline{x,y})| = (\mu - \lambda)\left(\frac{n(n-1)}{2} - 1\right).
\]
This contradiction of (B2) completes the proof.   
\end{proof}

 \section{Amalgamations and Detachments}\label{section:amalgamation}
 
In this section, we take a step towards the proofs of Theorems~\ref{thm:hamilton} and~\ref{thm:2factor} by formulating, in Corollary~\ref{cor:main}, conditions for the existence of solutions to Problems 1 and 2 in terms of the number of edges each colour class of the decomposition to be enclosed must contain.  
 We shall first need some definitions and auxiliary results.  
  
 A graph $H$ is an \emph{amalgamation} of a graph $G$ if there exists a surjection~$\phi$ from $V(G)$ onto $V(H)$ and a bijection $\psi$ between $E(G)$ and $E(H)$ such that $e \in E(G)$ joins $u$ and $v$ if and only if $\psi(e) \in E(H)$ joins $\phi(u)$ and $\phi(v)$.  The functions $\psi$ and $\phi$ are called \emph{amalgamation functions}.   
 We say that $G$ is a \emph{detachment} of $H$ if $H$ is an amalgamation of $G$; that is, if there exist amalgamation functions $\psi$ and $\phi$ as just defined.  We say that a vertex $v \in V(H)$ \emph{splits} into the set of vertices $\{ u \in V(G) \mid \phi(u)=v \in V(H)\}$.
 Let $\sigma: V(H) \rightarrow \mathbb{N}$ be a function.  Then 
 a \emph{$\sigma$-detachment} of $H$ is a detachment in which each vertex $v \in V(H)$ splits into $\sigma(v)$ vertices of $V(G)$. 
We note that as there is a bijection between the edges sets of a graph and a detachment, an edge-colouring of one naturally induces an edge-colouring of the other. If $a$ and $b$ are real numbers, then we write $a \approx b$ to mean $\lfloor a \rfloor \leq b \leq \lceil a \rceil$. For a graph $G$, let $\omega(G)$ denote the number of connected components of $G$, let $m_G(u, v)$ denote the number of edges joining vertices $u$ and $v$ in $G$, and let $\ell_G(w)$ denote the number of loops incident with vertex $w$ in $G$.

 The next lemma is a simplified version of a theorem of Bahmanian and Rodger~\cite{bahmanian1}.  

\begin{lemma}[\cite{bahmanian1}]\label{lem:bahmanian}
Let $H$ be a $k$-edge-coloured graph and let $\sigma: V(H) \rightarrow \mathbb{N}$ be a function such that for each $w \in V(H)$, $\sigma(w) = 1$ implies $\ell_H(w) = 0$. 
Then there exists a loopless $\sigma$-detachment $G$ of $H$ and an amalgamation function~$\phi$ from $V(G)$ onto $V(H)$ such that
\begin{itemize}
\item[\emph{(X1)}] $d_{G}(u) \approx d_H(w) / \sigma(w)$ for each $w \in V(H)$ and each $u \in \phi^{-1}(w)$;
\item[\emph{(X2)}] $d_{G(j)}(u) \approx d_{H(j)}(w) / \sigma(w)$ for each $w \in V(H)$, each $u \in \phi^{-1}(w)$ and each $j \in \{1, \dots, k\}$;
\item[\emph{(X3)}] $m_G(u, u') \approx \ell_H(w) / \binom{\sigma(w)}{2}$ for each $w \in V(H)$ with $\sigma(w) \geq 2$ and every pair of distinct vertices $u, u' \in  \phi^{-1}(w)$;
\item[\emph{(X4)}] $m_G(u,v) \approx m_H(w,z) /(\sigma(w)\sigma(z))$ for every pair of distinct vertices $w, z \in V(H)$, each $u \in \phi^{-1}(w)$ and each $v \in \phi^{-1}(z)$; and
\item[\emph{(X5)}] for $1 \leq j \leq k$, if $d_{H(j)}(w) / \sigma(w)$ is an even integer for each $w \in V(H)$, then $\omega(G(j)) = \omega(H(j))$.
\end{itemize}
\end{lemma}

Our next lemma is a complete multigraph analogue of~\cite[Theorem~2]{hilton1}. The argument is a straightforward adaptation of the proof of~\cite[Theorem 2]{hilton1} --- it merely suffices to apply Lemma~\ref{lem:bahmanian} instead of~\cite[Theorem 1]{hilton1} and this was observed by Bahmanian~\cite{bathesis}. For the sake of completeness, we include the proof.

\begin{lemma}\label{lem:hilton}
Let $k$, $m$, $n$, $r$ 
 and $\lambda$ be non-negative integers. A path decomposition $\A$ of $\lambda K_n$ of size $k$ in which  $r$ of the colour classes  are cycle-free can be embedded in a $2$-factorization of $\lambda K_{n+m}$ in which $r$ of the $2$-factors are Hamiltonian cycles if and only if
\begin{itemize}
\item[\emph{(Y1)}] $k = \lambda (n + m - 1) / 2$ is an integer, and
\item[\emph{(Y2)}] for $1 \leq i \leq k$, $\A(i)$ contains at most $m$ disjoint paths.
\end{itemize} 
\end{lemma}
 
 \begin{proof}
 \textit{Necessity}: Suppose $\A$ can be embedded in a $2$-factorization $\F$ of $\lambda K_{n+m}$ of size $k$.  For each $u \in V(\lambda K_{n+m})$, we have
 \begin{itemize}
 \item $\sum_{i=1}^k d_{\F(i)}(u) = \lambda( n+m - 1)$, and
 \item for $1 \leq i \leq k$, $d_{\F(i)}(u) = 2$. 
 \end{itemize}
 So (Y1) holds. For $1 \leq i \leq k$, the subgraph of $\F(i)$ induced by the vertices of $V(\lambda K_n)$ is $\A(i)$, a collection of paths.  Each of the two endvertices of each path is adjacent in $\F(i)$ to one of the vertices of $V(\lambda K_{m+n}) \setminus V(\lambda K_n)$ and each of these $m$ vertices is adjacent to at most two of the endvertices. So (Y2) holds. 
 
 \textit{Sufficiency}: We prove that $\A$ can be embedded in a $2$-factorization of $\lambda K_{n+m}$ in which $r$ of the $2$-factors are connected if (Y1) and (Y2) are satisfied.  We create a new graph $H$. It has vertex set $V(\lambda K_n) \cup \{v\}$, where $v$ is a new vertex and its edge set contains 
 \begin{itemize}
 \item the edges of $\lambda K_n$,
 \item for each $u \in V(\lambda K_n)$, $\lambda m$ edges joining $v$ to $u$, and
 \item $\lambda \displaystyle{\binom{m}{2}}$ loops on $v$. 
 \end{itemize}
 
 \begin{claim}\label{claim:1}
 There exists a $k$-edge-colouring $\B$ of $H$  such that 
 \begin{itemize}
 \item[A.] the restriction of $\B$ to $H - v$ is $\A$,
 \item[B.] for $1 \leq i \leq k$,  $\omega(\B(i)) = 1$ if and only if $\A(i)$ is cycle-free,
 \item[C.] for $u \in V(H) \setminus \{v\}$, for $1 \leq i \leq k$, $d_{\B(i)}u = 2$, and
 \item[D.] $2m$ edges of each colour class are incident with $v$. 
 \end{itemize}
 \end{claim}
 
 Before we prove the claim, we show that it implies the lemma.   Let $\sigma$ be a function that maps $V(H)$ to $\mathbb{N}$ such that $\sigma(v) = m$ and, for each $u \in V(H) \setminus \{v\}$, $\sigma(u) = 1$. 
 Note that the degree of $v$ is 
 (counting each loop twice) 
     \begin{eqnarray*}
     \lambda m n + 2 \lambda \binom{m}{2} &=& \lambda m n + \lambda m(m - 1) \\ 
      &=& 2mk.
     \end{eqnarray*}

We apply Lemma~\ref{lem:bahmanian} to find a loopless $\sigma$-detachment $G$ of $H$ and an amalgamation function $\phi$ from $V(G)$ onto $V(H)$. By (X1), $d_G(u)= \lambda (m+n-1)$ for each $u \in \phi^{-1}(v)$. By (X3), $m_G(u, u') = \lambda$ for each pair of distinct vertices $u, u' \in \phi^{-1}(v)$. By (X4), $m_G(u, w) = \lambda$ for each $u \in V(H) \setminus \{v\}$ and each $w \in \phi^{-1}(v)$. Therefore $G = \lambda K_{n+m}$.  
Combining Claim~\ref{claim:1}.A, Claim~\ref{claim:1}.C, Claim~\ref{claim:1}.D and (X2), we have that the $k$-edge-colouring of $G$  yields a $2$-factorization $\B$ of $\lambda K_{n+m}$ of size $k$ such that $\A$ is embedded in $\B$. Finally Claim~\ref{claim:1}.B and (X5) imply that exactly $r$ of the $2$-factors in $\B$ are connected.  It only remains to prove Claim~\ref{claim:1}.

It may be assumed that Claim~\ref{claim:1}.A holds.  Notice that, for each $u \in V(H) \setminus \{v\}$, for $1 \leq i \leq k$, we have
  $d_{\A(i)}(u) \leq 2$ and, by (Y1), $d_{H}(u) = \lambda (m+n - 1)  = 2k$. Thus we can colour the non-loop edges incident with $v$ such that two edges of each colour class are incident with $u$. So Claim~\ref{claim:1}.C holds. Notice also from (Y2) that the number of edges between $v$ and $V(H) \setminus \{v\}$  in a given colour class is even and is at most $2m$ (since for each path in $\A(i)$ two of these edges are coloured $i$).   Thus, since the degree of $v$ is $2mk$, the loop edges incident with $v$ can be coloured in such a way that $2m$ edges of each colour class are incident with $v$. So Claim~\ref{claim:1}.D holds. As Claim~\ref{claim:1}.B is clearly satisfied the proof is complete. 
 \end{proof}

 \begin{corollary}\label{cor:main}
 Let $k$, $n$, $m$,  $\lambda$, and $\mu$ be positive integers such that $\mu > \lambda$. A path decomposition $\A$ of $\lambda K_n$ of size $k$  can be enclosed in a $2$-factorization $\F$ of $\mu K_{n+m}$ if and only if
 \begin{itemize}
 \item[\emph{(W1)}] $k = \mu (n + m - 1) / 2$ is an integer, and 
 \item[\emph{(W2)}] $\A$ can be enclosed in a path decomposition $\B$ of $\mu K_n$  of size $k$ in which, for $1 \leq i \leq k$, the number of edges in $\B(i)$ is at least $n - m$.
 \end{itemize}
 Moreover, if $\A$ is strong then, in addition, $\F$ is a Hamiltonian decomposition and $\B$ is strong. 
 \end{corollary}

 \begin{proof}
 It suffices to show that, by Lemma~\ref{lem:hilton} (with $r=k$ if $\A$ is strong), that conditions (Y2) and (W2) are equivalent. This follows from the fact that, for $1 \leq i \leq k$, $\B(i)$ contains at least $n-m$ edges if and only if it contains at most $n - (n - m) = m$ disjoint paths.   
 \end{proof}
 
\section{Extending Partial Path Decompositions}\label{section:lemmas}

 A partial decomposition of $\mu K_n$ is said to be \emph{strict} if it is a decomposition of a \emph{proper} subgraph of $\mu K_n$.  
In this section, we describe in Lemmas~\ref{lem:extension-strong} and~\ref{lem:extension-weak} conditions that guarantee when it is possible to extend a strict partial path decomposition of $\mu K_n$ into a partial path decomposition that contains one more edge.

 \begin{lemma}\label{lem:extension-strong}
 Let $k$, $n$, $\lambda$ and $\mu$ be positive integers such that $\mu > \lambda$ and $k \geq \mu(n-1) - 1$. Let $\A$ be a strong path decomposition of $\lambda K_n$ of size $k$. Suppose that $\A$ is enclosed in a strict partial strong path decomposition $\A'$ of $\mu K_n$ of size $k$.  Then $\A$ can be enclosed in a partial strong path decomposition $\A^*$ of $\mu K_n$ of size $k$ whose colour classes are the same size as those of $\A'$ except for one that contains one more edge.
   \end{lemma}
   
   \begin{proof}
   Let $L$ be the set of edges of $\mu K_n \setminus \lambda K_n$ that are not in $\A'$.  Our aim is to show that at least one edge $e$ of $L$ can be assigned a colour $j$ in such a way that the resulting decomposition is a strong path decomposition. We say that the edge $e$ and the colour $j$ are  \emph{compatible}.  
 We remark that before finding $e$, we may first amend $\A$, but will never reduce the number of edges in a colour class.

   For a vertex $v$, for $0 \leq i \leq 2$, in $\A'$ let $C_v^i$ denote the set of colours that occur exactly $i$ times on the edges incident with $v$.
   
   Let $e$ be an edge of $L$, and let its incident vertices be $u$ and $v$. If there is a colour $j$ that belongs to  $C_u^0 \cap (C_v^0 \cup C_v^1)$ or $C_v^1 \cap (C_u^0 \cup C_u^1)$ then $e$ and $j$ are compatible and we are done. Otherwise, $C_u^0 \subseteq C_v^2$ and $C_v^0 \subseteq C_u^2$ hold.  
   
   In fact, we make the following stronger claim.
   
   \begin{claim} $C_u^0 = C_v^2$, $C_v^0 = C_u^2$ and $C_u^1 = C_v^1$
\end{claim}   
  \noindent  The number $f$ of edges incident with $u$ that are not in $L$ is at most $\mu (n-1) - 1$ (one less than its degree). Thus, since $k \geq f$, for each pair of edges incident with $u$ that are coloured alike, there is a colour that does not occur on the edges incident with $u$. Hence $|C_u^2| \leq |C_u^0|$. Similarly, we find that $|C_v^2| \leq |C_v^0|$. If $C_v^0 \subsetneq C_u^2$ then $|C_v^2| \leq |C_v^0| < |C_u^2| \leq |C_u^0|$, a contradiction. Therefore $C_u^0 = C_v^2$ and $C_v^0 = C_u^2$ implying $C_u^1 = [k] \setminus (C_u^0 \cup C_u^2) = [k] \setminus (C_v^2 \cup C_v^0) = C_v^1$.
 The claim is proved.

    Let $e'$ be a $uv$-edge in $\lambda K_n$ and let $a$ be the colour of $e'$ in $\A$.  Since $C_u^2 = C_v^0$, it follows that $a \in C_u^1$.  Hence $\A'(a)$ contains at least two disjoint paths $Q_1$, $Q_2$ where $Q_1$ is $e'$. Let $w$ be an endpoint of $Q_2$, and let $e''$ be a $uw$-edge of $\mu K_n \setminus \lambda K_n$. If $e'' \in L$,  then $e''$ and $a$ are clearly compatible. Else the colour $b \not=a$ of $e''$ either belongs to $C_x^2  = C_y^0$  or to $C_x^1 = C_y^1$. In either case recolouring $e''$ to $a$  will give a strong path decomposition in which $e$ and $b$ are compatible. This completes the proof. 
   \end{proof}
   
 \begin{lemma}\label{lem:extension-weak}
Let $k$, $n$, $m$ and $\mu$ be positive integers such that $m \geq n-2$, $n \geq 4$ if $m=n-2$, and $k \geq \mu (n+m-1) / 2$. Let $\A$ be a strict partial path decomposition of $\mu K_{n}$ of size $k$ containing at most $\mu (m-1) / 2$ 2-factors. Then a single edge can be added to one of the colour classes of $\A$ to give a partial path decomposition $\A'$ of $\mu K_n$ of size $k$ that also contains at most $\mu (m-1) / 2$ 2-factors. 
  \end{lemma}

  \begin{proof}
  Let $e=uv$ be an edge of $\mu K_n \setminus \A$.
To prove the lemma we will show that we can choose a colour $j$ to assign to $e$ to obtain another partial path decomposition $\A'$, and that if the number of 2-factors in $\A$ is $\mu (m-1) / 2$, this assignment does not create an additional 2-factor.
  
Suppose that the number of 2-factors in $\A$ is less than $\mu (m-1) / 2$.  Then it is enough to find a colour that appears on no more than one edge incident with $u$ or $v$ (it is possible that colouring $e$ with $j$ will create a 2-cycle and that this will complete a further 2-factor).
Vertices $u$ and $v$ are joined by $\mu$ edges, and they are each joined to the other $n-2$ vertices by $\mu$ edges.  Hence the total number of coloured edges incident with $u$ or $v$ is, recalling that $e$ is not coloured, at most $\mu(2n-3)-1$. So the total number of colours that appear on at least two of these edges is at most
\[ \frac{\mu(2n-3)}{2} - \frac{1}{2}.
\]
But, by the bounds on $k$ and $m$, we have that   
\[ k \geq \frac{\mu(2n-3)}{2},
\]
and so there is at least one colour that is not used on two of the edges and this colour can be assigned to $e$.

Suppose now that the number of $2$-factors in $\A$ is $\mu (m-1) / 2$. So we need to find a colour to assign to $e$ that does not complete a further $2$-factor. Then it is enough to find a colour that appears on at most one edge incident with either $u$ or $v$ but not both. We define a partition of the $2$-factors:
\begin{itemize}
\item $F_1$ is the set of $2$-factors that contain a $2$-cycle consisting of two $uv$-edges.
\item $F_2$ is the set of $2$-factors that contain exactly one $uv$-edge.
\item $F_3$ is the set of $2$-factors that do not contain a $uv$-edge. 
\end{itemize}

So, from the set of edges incident with $u$ or $v$, for $1 \leq i \leq 3$, each $2$-factor of $F_i$ contains $i+1$ edges. Note by assumption that
\begin{equation}\label{eq:1}
\sum_{i=1}^3{|F_i|} = \frac{\mu(m-1)}{2}.
\end{equation}

We also define a partition of some colour classes that are not $2$-factors:
\begin{itemize}
\item $Q_1$ is the subset of colour classes that are not $2$-factors that contain a $uv$-edge. 
\item $Q_2$ is the subset of colour classes that are not $2$-factors that do not contain a $uv$-edge but contain at least two edges incident with $u$ or $v$.  
\end{itemize}
So, from the set of edges incident with $u$ or $v$, for $1 \leq i \leq 2$, each graph in $Q_i$ contains at least $i$ edges. 

Let 
\[ k^\ast = |F_1| + |F_2| +|F_3| +|Q_1| +|Q_2|. 
\]
Note that a colour class not in $F_1$, $F_2$, $F_3$, $Q_1$, $Q_2$ contains at most one edge incident with $u$ and $v$ and this edge is not a $uv$-edge.  Our aim is to show that there is at least one such class; that is, that $k^\ast < k$.

Considering the $\mu-1$ $uv$-edges distinct from $e$, we have
\begin{equation}\label{eq:2}
2|F_1| + |F_2| + |Q_1| \leq \mu - 1.
\end{equation}
Considering all edges incident with $u$ or $v$ gives us
\begin{eqnarray}\label{eq:main}
2|F_1| + 3|F_2| +4|F_3| +|Q_1| +2|Q_2| &\leq & \mu(2n-3) - 1 \nonumber \\
2k^\ast + |F_2| +2|F_3| -|Q_1|  &\leq & \mu(2n-3) - 1,\nonumber
\end{eqnarray}
and adding (\ref{eq:2}) gives us
\[
2k^\ast + 2|F_1|+2|F_2|+2|F_3| \leq \mu(2n-2)-2. 
\]
From (\ref{eq:1}), we obtain
\[2k^\ast + \mu(m-1) \leq \mu(2n-2)-2. 
\]
So
\[2k^\ast  \leq \mu((2n-2) - (m-1))-2<2k. 
\]
Therefore $k^\ast<k$ and there is a colour class that is in none of $F_1$, $F_2$, $F_3$ nor $Q_1$, $Q_2$ and the proof is complete.
  \end{proof}

\section{Proofs of Theorems~\ref{thm:hamilton} and \ref{thm:2factor}}\label{section:main}
  
  \noindent \textbf{Proof of Theorem~\ref{thm:hamilton}.}
  \textit{Necessity:} Suppose a strong path decomposition $\A$ of $\lambda K_n$ of size $k$ can be enclosed in a Hamiltonian decomposition of $\mu K_{n+m}$.  It may be assumed, by Corollary~\ref{cor:main}, that conditions (W1) and (W2) hold. Thus (W2)  implies $\lambda K_n$ is strongly $(n-m)$-extendible with respect to $(\A, \mu K_n)$. Thus the necessity of (M2) and (M3) follows from Propositions~\ref{prop:1} and~\ref{prop:2}. Finally (W1) implies (M1) holds. 
  
  \textit{Sufficiency}: To prove sufficiency it remains to show, by Corollary~\ref{cor:main}, that conditions (M1)--(M3) imply conditions (W1) and (W2). Trivially (W1) is implied by  (M1). By Propositions~\ref{prop:1} and~\ref{prop:2} and (M2) and (M3), $\lambda K_n$ is strongly $(n-m)$-extendible with respect to $(\A, \mu K_n)$. Let $\A'$ be the resulting partial strong path decomposition in which each colour class contains at least $n-m$ edges. Since $k \geq \mu(n-1) - 1$ whenever (i), (ii) or (iii) holds, we can repeatedly apply Lemma~\ref{lem:extension-strong} starting from $\A'$ to obtain a strong path decomposition in which each colour class contains at least $n - m$ edges. So (W2) is satisfied and the proof is complete.   
  \qed
 
 \medskip
  
 \noindent \textbf{Proof of Theorem~\ref{thm:2factor}.}
  \textit{Necessity}: Suppose a path decomposition $\A$ of $\lambda K_n$ of size $k$ can be enclosed in a $2$-factorization of $\mu K_{n+m}$. It may be assumed, by Corollary~\ref{cor:main}, that conditions (W1) and (W2) hold.   
  Thus (W2) implies $\lambda K_n$ is $(n-m)$-extendible with respect to $(\A, \mu K_n)$. The necessity of (N3) then follows from Propositions~\ref{prop:1} and~\ref{prop:3}. Also (W1) trivially implies (N1) holds. To prove the necessity of (N2), we argue by contradiction. Suppose that the number of $2$-factors in $\A$ is  $ \frac{\mu(m-1) + 2}{2}$.
   Since $\lambda K_n$ is $(n-m)$-extendible with respect to $(\A, \mu K_n)$, the colour classes that are not $2$-factors of $\lambda K_n$ must each contain at least $n-m$ edges in $\mu K_n$ and so their number cannot exceed $\frac{1}{n-m}$ of the number of edges of $\mu K_n$ not used by the $2$-factors.   
  Notice that the latter is equal to
\begin{eqnarray*}
\mu \left(\frac{n(n-1)}{2}\right) - n\left(\frac{\mu(m-1) + 2}{2}\right)
= (n-m)  \frac{\mu n}{2} - n \\
\end{eqnarray*}
while the former is equal to
\begin{eqnarray*}
k - \frac{\mu(m-1) + 2}{2}
&=& \frac{\mu n - 2}{2} \\
&>& \left((n-m)  \frac{\mu n}{2} - n\right) / (n-m)
 \end{eqnarray*}
 for $m \geq n-2$. This contradiction proves that (N2) holds.

  \textit{Sufficiency}: To prove sufficiency it remains to show, by Corollary~\ref{cor:main}, that conditions (N1)-(N3) imply conditions (W1) and (W2). Trivially (W1) is implied by  (N1). 
 Propositions~\ref{prop:1} and~\ref{prop:3} and (N3) imply that $\lambda K_n$ is $(n-m)$-extendible with respect to $(\A, \mu K_n)$. Let $\A'$ be the resulting partial path decomposition in which each colour class contains at least $n-m$ edges. If $n=3$ and $m=1$ then $\A'$ is the required path decomposition since $2k = 3 \mu$ is the number of edges in $\mu K_3$. In all other cases, (N2) allows us to repeatedly apply Lemma~\ref{lem:extension-weak} starting from $\A'$ to obtain a path decomposition in which each colour class contains at least $n-m$ edges. So (W2) is satisfied and the proof is complete. \qed

 \bibliography{bibliography}{}

\begin{thebibliography}{10}

\bibitem{Asplund0}
J.~Asplund.
\newblock 5-cycle systems of {$(\lambda+ m) K_{v+ 1} - \lambda K_v$ and
  $\lambda K_{v+ u}- \lambda K_v$}.
\newblock {\em Discrete Mathematics}, 338(5):766--783, 2015.

\bibitem{ceven}
J.~Asplund.
\newblock {$m$}-cycle packings of {$(\lambda+ \mu) K_{v+ u}-\lambda K_v :$}
  {$m$} even.
\newblock {\em arXiv}, 1511.09301, 2015.

\bibitem{asplundone}
J.~Asplund, C.~A. Rodger, and M.~S. Keranen.
\newblock Enclosings of {$\lambda$}-fold {$5$}-cycle systems: Adding one
  vertex.
\newblock {\em Journal of Combinatorial Designs}, 22(5):196--215, 2014.

\bibitem{Asplund2}
J.~Asplund, C.~A. Rodger, and M.~S. Keranen.
\newblock Enclosings of {$\lambda$}-fold 5-cycle systems for {$u=2$}.
\newblock {\em Discrete Mathematics}, 338(5):743 -- 765, 2015.

\bibitem{bathesis}
M.~A. Bahmanian.
\newblock {\em Amalgamations and Detachments of Graphs and Hypergraphs}.
\newblock PhD thesis, Auburn University, 2012.

\bibitem{bahmanian1}
M.~A. Bahmanian and C.~A. Rodger.
\newblock Multiply balanced edge colorings of multigraphs.
\newblock {\em Journal of Graph Theory}, 70(3):297--317, 2012.

\bibitem{amalgamation4}
M.~A. Bahmanian and M.~{\v{S}}ajna.
\newblock Decomposing complete equipartite multigraphs into cycles of variable
  lengths: The amalgamation-detachment approach.
\newblock {\em Journal of Combinatorial Designs}, 2014.

\bibitem{c31}
C.~J. Colbourn, R.~C. Hamm, and A.~Rosa.
\newblock Embedding, immersing, and enclosing.
\newblock {\em Congressus Numerantium}, 47:229--236, 1985.

\bibitem{diestel}
R.~Diestel.
\newblock {\em Graph Theory}.
\newblock Springer-Verlag, Heidelberg, 4th edition, 2010.

\bibitem{hilton1}
A.~J.~W. Hilton.
\newblock Hamiltonian decompositions of complete graphs.
\newblock {\em Journal of Combinatorial Theory, Series B}, 36(2):125--134,
  1984.

\bibitem{amalgamation1}
A.~J.~W. Hilton, M.~Johnson, C.~A. Rodger, and E.~B. Wantland.
\newblock Amalgamations of connected k-factorizations.
\newblock {\em Journal of Combinatorial Theory, Series B}, 88(2):267--279,
  2003.

\bibitem{cvarying}
D.~Horsley and R.~A. Hoyte.
\newblock Decomposing {$K_{u+ w}-K_u$} into cycles of various lengths.
\newblock {\em arXiv}, 1603.03908, 2016.

\bibitem{c32}
S.~P. Hurd, P.~Munson, and D.~G. Sarvate.
\newblock Minimal enclosings of triple systems {I}: adding one point.
\newblock {\em Ars Combinatoria}, 68:145--159, 2003.

\bibitem{c33}
S.~P. Hurd and D.~G. Sarvate.
\newblock Minimal enclosings of triple systems {II}: increasing the index by 1.
\newblock {\em Ars Combinatoria}, 68:263--282, 2003.

\bibitem{amalgamation2}
M.~Johnson.
\newblock Amalgamations of factorizations of complete graphs.
\newblock {\em Journal of Combinatorial Theory, Series B}, 97(4):597--611,
  2007.

\bibitem{amalgamation3}
C.~{\relax St}. J.~A. Nash-Williams.
\newblock Amalgamations of almost regular edge-colourings of simple graphs.
\newblock {\em Journal of Combinatorial Theory, Series B}, 43(3):322--342,
  1987.

\bibitem{c42}
N.~A. Newman.
\newblock 4-cycle decompositions of {$(\lambda+ m) K_{v+ u} \setminus \lambda
  K_v$}.
\newblock {\em Designs, Codes and Cryptography}, 75(2):223--235, 2015.

\bibitem{c4}
N.~A. Newman and C.~A. Rodger.
\newblock Enclosings of {$\lambda$}-fold 4-cycle systems.
\newblock {\em Designs, Codes and Cryptography}, 55(2-3):297--310, 2010.

\end{thebibliography}
\bibliographystyle{abbrv}
 
\end{document}